\documentclass[12pt]{amsart}

\allowdisplaybreaks[1]
\usepackage{enumerate}
\usepackage{graphicx}
\usepackage{allrunes}
\usepackage{amsmath}
\usepackage{amssymb}
\usepackage{amsfonts}
\usepackage[dvipsnames]{xcolor}
\usepackage{hyperref}
\usepackage{skak}
\allowdisplaybreaks
\makeindex

 \newtheorem{theorem}{Theorem}[section]
 \newtheorem{corollary}[theorem]{Corollary}
 \newtheorem{lemma}[theorem]{Lemma}

\newtheorem{observation}[theorem]{Observation}
\theoremstyle{definition}

\theoremstyle{remark}

\newtheorem{fact*}{Fact}

\newcommand\dd{\mathrm d}

\newcommand{\hilbert}{\mathcal{H}}

\renewcommand\h{\mathcal{H}}

\newcommand{\BH}{\mathcal{B}(\mathcal{H})}

\newcommand{\abs}[1]{\left\vert#1\right\vert}

\newcommand{\inv}{^{-1}}

\newcommand{\til}{\raise.17ex\hbox{$\scriptstyle\mathtt{\sim}$}}

\newcommand\beq{\begin{equation}}

\newcommand\eeq{\end{equation}}

\newcommand\bbm{\begin{bmatrix}}
\newcommand\ebm{\end{bmatrix}}

\newcommand{\bbms}{\left[ \begin{smallmatrix}}
\newcommand{\ebms}{\end{smallmatrix} \right]}

\newcommand{\bpm}{\begin{pmatrix}}
\newcommand{\epm}{\end{pmatrix} }
\numberwithin{equation}{section}

\newlength{\Mheight}
\newlength{\cwidth}

\newcommand{\dfn}[1]{{\bf #1}\index{#1}}

\newcommand{\MU}[1]{\mathcal{M}(#1)}

\newcommand{\vex}[1]{\overrightarrow{#1}}

\newcommand{\BK}{\mathcal{B}(\mathcal{K})}
\newcommand{\KK}{\mathcal{K}}

\title[Free universal monodromy and plurisubharmonicity]{Noncommutative free universal monodromy, pluriharmonic conjugates, and plurisubharmonicity  }
\author[J. E. Pascoe]{
J. E. Pascoe
}
\address{Department of Mathematics\\
1400 Stadium Rd\\
  University of Florida\\
 Gainesville, FL 32611}
\email[J. E. Pascoe]{pascoej@ufl.edu}

\date{\today}


\begin{document}

\begin{abstract}
We show that the monodromy theorem holds on arbitrary connected free sets for noncommutative free analytic functions. Applications are numerous--
pluriharmonic free functions have globally defined pluriharmonic conjugates, locally invertible functions are globally invertible, and there is no nontrivial cohomology theory arising from analytic continuation on
connected free sets.
We describe why the Baker-Campbell-Hausdorff formula has finite radius of convergence in terms of monodromy, and solve a related problem of Martin-Shamovich.
We generalize the Dym-Helton-Klep-McCullough-Volcic theorem-- a uniformly real analytic free noncommutative function 
is plurisubharmonic if and only if it can be written as a composition of a convex function with an analytic function. The decomposition is essentially unique. The result is first established locally, and then Free Universal Monodromy implies the global result. Moreover, we see that plurisubharmonicity is a geometric property-- a real analytic free function plurisubharmonic on a neighborhood is plurisubharmonic on the whole domain. We give an analytic Greene-Liouville theorem, an entire free plurisubharmonic function  is a sum of hereditary and antihereditary squares.
\end{abstract}

\maketitle

\section{Introduction}
	Let $U\subseteq \mathbb{C}$ be open. A function $u: U \rightarrow \mathbb{R}\cup\{-\infty\}$ is said to be \dfn{subharmonic} if  $\Delta u \geq 0.$ (Here, if $u$ is not $C^2,$ one could have taken the Laplacian of $u$ in the sense of distributions.)
	In the early twentieth century, Riesz showed \cite{rie23, riesz1926, riesz1930} that a subharmonic function which is bounded above by some harmonic function is of the form
		\beq \label{rieszrep} u(z)= v(z)+\int_{\mathbb{C}} \log \abs{z-w} \dd\mu(w)\eeq
	for some harmonic function $v$ and positive Borel measure $\mu$ on $\mathbb{C}.$

	Let $U \subseteq \mathbb{C}^n$ be open.
	A function $u: U \rightarrow \mathbb{R}\cup\{-\infty\}$ is said to be \dfn{plurisubharmonic} if the complex Hessian $[\frac{\partial^2}{\partial \overline{z_i} \partial z_j} u]_{i,j}\geq 0.$ (Given a square matrix
	$M$ we say $M\geq 0$ if $M$ is positive semidefinite.)
	Plurisubharmonic functions are locally the composition of a convex function with and analytic function, and thus are foundational to the notion in several complex variables of pseudoconvexity \cite{oka42, dangelo}.

	The \dfn{monodromy theorem} in complex variables says that if a function analytically continues along every path in a simply connected domain, then it analytically continues to the whole domain.

	We treat monodromy, pluriharmonic functions and plurisubharmonic functions in the case of free noncommutative function theory. Free functions are modeled on
	\dfn{noncommutative polynomials}, expressions in terms of some noncommuting indeterminates. For example, both
		$$7x_1x_2-x_2x_1, 3x_1^{82}x_2x_1-x_2$$
	are free polynomials in two variables. If we allow inversions, we get the \dfn{free rational expressions}. For example,
		$$(1-3x_1^{-82}x_2x_1)^{-1}-x_2$$
	is a noncommutative rational expression.
	Continuing this line of generalization, one could consider noncommutative power series $$\sum c_\alpha x^\alpha$$ where $\alpha$ ranges over all words in some noncommuting letters. For example,
	the Baker-Campbell-Hausdorff formula is such an expression of classical importance in Lie theory \cite{CBHobstructions}.	
	Such expressions can naturally be evaluated over matrices where they converge. For most of our discussion, the reader could imagine that free noncommutative function is something that is given locally by a noncommutative power series without much loss of fidelity.

	\subsection{Free noncommutative functions}
	Let $R$ be a topological vector space over $\mathbb{C}.$		
	We define the matrix universe over $R$ to be
		$$\MU{R}= \bigcup^{\infty}_{n=1}M_n(\mathbb{C})\otimes R.$$
	Whenever $R \cong \mathbb{C}^d$ is finite dimensional, $\MU{R}$ essentially consists of $d$-tuples of matrices.
	
	A \dfn{free set} $D\subseteq \MU{R}$ satisfies the following:
	\begin{enumerate}
		\item $X, Y \in D \Rightarrow X \oplus Y \in D,$ 
		\item If $X \in D$ and $V$ is a unitary, then $V^*XV \in D.$
	\end{enumerate}
	Given $D$ a free set, we let $D_n$ denote $D \cap M_n(\mathbb{C})\otimes R.$
	We say a free set is \dfn{open} if each $D_n$ is open. We say a free set is \dfn{connected} if each $D_n$ is path connected.	
	When $R$ is a subspace of an operator algebra, we say $D$ is \dfn{uniformly open} if for every $X\in D_m$ there is an $\epsilon > 0$ such that
	$B(\oplus^n X,\epsilon) \subseteq D_{mn}$ for all $n.$

	Let $D \subseteq \MU{R_1}.$
	We define a \dfn{real free function} $f: D \rightarrow \MU{R_2}$ to be a function satisfying the following,
	\begin{enumerate}
		\item $f(D_n) \subseteq \MU{R_2}_n$
		\item $f(X \oplus Y) =f(X)\oplus f(Y),$ 
		\item If $V$ is unitary, then $f(V^*XV) =V^*f(X)V.$ 
	\end{enumerate}
	Given a free function $f$, we write $f_n$ for $f|_{D_n}.$
	We say $f$ is \dfn{analytic} if each $f_n$ is analytic.
	We say $f$ is \dfn{uniformly analytic} if $f$ is analytic and for every $X$ there is a uniformly open free set $D_X \subseteq D$ such that $X\in D_X$ and $f|_{D_X}$ is bounded.
	We similarly define \dfn{real analytic} and \dfn{uniformly real analytic.} Importantly, if $R_1 \cong \mathbb{C}^d$ is finite dimensional, analytic free functions have a noncommutative power series
	in variables $z_1, \ldots z_d$,
	and real analytic free functions have a power series in terms of the letters $z_1, \ldots, z_d$ and their adjoints $z_1^*, \ldots z_d^*.$ An alternative characterization of analytic free functions is that
	they are locally bounded and preserve arbitrary similarites, that is $f(S^{-1}XS) = S^{-1}f(X)S$ where both sides of the formula are defined \cite{KlemTerHorst,PTDPick,helkm11, vvw12}.

	For a real analytic free function $f,$ we denote the \dfn{derivative at $Z$ in a direction $H$} by $$Df(Z)[H] = \frac{\partial }{\partial z}f(Z+zH)\big|_{z=0}.$$
	Here, $\frac{\partial }{\partial z}$ is taken in the sense of complex variables. Given that $\frac{\partial }{\partial z} z = 1, \frac{\partial }{\partial z} \overline{z} =0$
	and $\frac{\partial }{\partial z}(fg) = \left(\frac{\partial }{\partial z}f\right)g + f\left(\frac{\partial }{\partial z}g\right),$ we see that
		$$DZ_i[H] = H_i, DZ_i^*[H] = 0,$$ $$D(fg)(Z)[H] = Df(Z)[H]g(Z)+f(Z)Dg(Z)[H].$$
	Moreover, this allows us to take the derivative of a noncommutative polynomial or power series with relative ease.
	For example, let 
		\beq \label{derivexam} f(Z) = Z_1 +Z_1^*+Z_2Z_2^*Z_1+ Z_2^*Z_2Z_1^*.\eeq
	We have that 
		$$Df(Z)[H]= H_1+ H_2Z_2^*Z_1 + Z_2Z_2^*H_1+Z_2^*H_2Z_1^*.$$
	Similarly, we denote the \dfn{conjugate derivative at $Z$ in a direction $H$} by $$D^*f(Z)[H] = \frac{\partial}{\partial \overline{z}}f(Z+zH)\big|_{z=0}.$$
	We have the corresponding relations for the conjugate derivative,
		$$D^*Z_i[H] = 0, D^*Z_i^*[H] = H_i^*,$$ $$D^*(fg)(Z)[H] = D^*f(Z)[H]g(Z)+f(Z)D^*g(Z)[H].$$
	Notably, for any analytic function $f,$ $D^*f \equiv 0,$ and $D^*f^*(Z)[H] = (Df(Z)[H])^*.$
	For the function $f$ in \eqref{derivexam}, we have that
		$$D^*f(Z)[H]= H_1^*+ H_2^*Z_2Z_1^* + Z_2^*Z_2H_1^*+Z_2H_2^*Z_1.$$
	We denote the \dfn{complex Hessian at $Z$ in the direction $H$} of a real analytic free function $u$ by $$\Delta u (Z)[H]=(\Delta_z u(Z+zH))|_{z=0}
	=\left(\frac{\partial}{\partial \overline{z}}\frac{\partial }{\partial z}  u(Z+zH)\right)\bigg|_{z=0}$$ 
	where $\Delta_z$ is the classical Laplacian with respect to $z.$
	Essentially, $$\Delta f(Z)[H] = D^*Df(Z)[H][H] =DD^*f(Z)[H][H].$$
	For the function $f$ in \eqref{derivexam}, we have that
		$$\Delta f(Z)[H]= H_2H_2^*Z_1 + Z_2H_2^*H_1+H_2^*H_2Z_1^*+Z_2^*H_2H_1^*.$$
	Suppose $f: D \rightarrow \MU{R_2}$ where $R_2$ is a subspace of an operator algebra.	
	We say a self-adjoint-valued real analytic free function is \dfn{plurisubharmonic} if  $\Delta f (Z)[H] \geq 0$ for all $X$ in the domain and for all directions $H.$

	For example, the function $f(Z)= Z^*Z$ is plurisubharmonic as $\Delta f(Z) = H^*H.$ However, consider the function
		$$g(z)= \log |Z| = \frac{1}{2}\log(Z^*Z).$$
	Consider
	\begin{align*}
		\Delta g\bpm 1 & 0 \\ 0 & 1 \epm \bbm 0 & 1 \\ 0 & 0 \ebm&= \Delta_z \log \left|\bpm 1 & z \\ 0 & 1 \epm\right|\bigg|_{z=0}\\
		&= \Delta_z \frac{1}{2}\log \bpm 1 & z \\ \overline{z} & 1+|z|^2\epm\bigg|_{z=0}\\
		&= \Delta_z \frac{1}{2}\sum \frac{(-1)^{n+1}}{n} \bpm 0 & z \\ \overline{z} & |z|^2\epm^n\bigg|_{z=0}\\
		&= \bpm -\frac{1}{4} & 0 \\ 0 & \frac{1}{4} \epm.
	\end{align*}
	Therefore, $g$ is not plurisubharmonic in contrast to the classical case as in \eqref{rieszrep}, and theory of Brown measure, which relies crucially on the fact that 
	$\tau(\log |A-z|)$ is subharmonic for as a function of $z\in\mathbb{C}$ whenever $A$ lies in a tracial von Neumann algebra with trace $\tau$ \cite{brownmeasure, Mingo2017}.

	We denote the \dfn{real derivative at $X$ in the direction $H$} by $D_\mathbb{R}(X)[H] = \frac{d}{dt}f(X+tH)|_{t=0}.$ (Here, we are taking the derivative
	with respect to $t$ as a real variable.)
	We have the corresponding relations for the real derivative,
		$$D_\mathbb{R}Z_i[H] = H_i, D_\mathbb{R}Z_i^*[H] = H_i^*,$$ $$D_\mathbb{R}(fg)(Z)[H] = D_\mathbb{R}f(Z)[H]g(Z)+f(Z)D_\mathbb{R}g(Z)[H].$$
	We denote the \dfn{real second derivative at $X$ in the direction $H$} by $D_\mathbb{R}^2(X)[H] = \frac{d^2}{dt^2}f(X+tH)|_{t=0}.$
	We say $f$ is \dfn{convex} if $D_\mathbb{R}^2(X)[H]\geq 0$ for all $X$ in the domain and for all directions $H.$

	Plurisubharmonic free functions have been analyzed in the case of polynomials and rational functions. Greene, Helton and Vinnikov \cite{GHV}, followed by Greene \cite{greene}
	showed that for noncommutative polynomials, all plurisubharmonic functions are hereditary plus anti-hereditary, that is there are analytic polynomials $h_1, \ldots, h_k$ and $g_1, \ldots g_l$
	such that $$p = \sum h_ih_i^* + \sum g_i^*g_i.$$ Moreover, being locally plurisubharmonic was enough to conclude global plurisubharmonicity.
	Dym, Klep, Helton, McCullough and Volcic, showed that noncommutative rational plurisubharmonic functions can be written as the composition of a convex function with an analytic one \cite{DHKMV}.
	Convex free functions are of the form
		\beq \label{butterflyrealization}
			f(X) = a_0 + L(X) + \Lambda(X)^* (I -  \Gamma(X))\inv \Lambda(X)
		\eeq
	where $\h$ is a Hilbert space, $L:R_1\rightarrow \BK$, $\Lambda:R_1 \rightarrow \mathcal{B}(\KK,\h)$ and $\Gamma: R_1 \rightarrow \BH$ are completely bounded $\mathbb{R}$-linear maps,
	where $L$ and $\Gamma$ are self-adjoint valued,
	as was established in the rational case by Helton, McCullough and Vinnikov \cite{heltonbutterfly} and in the analytic case by Pascoe and Tully-Doyle \cite{royalroad}.
	By formally taking the second derivative of a function with such a representation, one sees that it must be convex wherever $I-\Gamma(Z)$ is positive definite and all the elements are defined. 
	That is, letting $f$ be as in \eqref{butterflyrealization}, $R(X)=(I -  \Gamma(X))\inv,$ and $v(X)=\Gamma(H)R(X)\Lambda(X)+\Lambda(H),$ we have that
		$$D^2_\mathbb{R} f(X)[H] = 2v(X,H)^*R(X)v(X,H).$$
	In the globally defined case, including noncommutative polynomials,
	it is known that $\Gamma \equiv 0$ and therefore that all such functions are quadratic noncommutative polynomials \cite{helmconvex04,dhmconvex,hhlm2008,hptdvconvex}.
	Other realizations for convex functions on certain domains have been established in \cite{palfia}.
	Convex free functions are related to the burgeoning study of matrix convex sets \cite{KennyD, freelasserre, effwink} and their change of variables theory \cite{ahkm,helkm11, hkms09}. A clear motivation of \cite{DHKMV}
	was to develop a corresponding theory of pseudoconvex sets,
		with an eye towards repeating previous success for convexity in systems theory and engineering \cite{HKMBook, CHSY, DeOliveira2009, SIG}.
		Namely, one wants a general theory for when a nonconvex problem is equivalent to a convex problem, hence the motivation for pseudoconvexity and 
		plurisubharmonicity.
		We also note that various realizations for various classes of functions permeate noncommutative free function theory,
	see \cite{royalroad, bgm06, bgm05, bmv18, agmc_gh, palfia, pptd19, pastdcauchy}.
	
	We prove the following representation formula for plurisubharmonic free functions.
	\begin{theorem}\label{introrealization}
		Let $R_1$ be a finite dimensional vector space.
		Let $R_2$ be a subspace of operators.
		Let $D \subseteq \MU{R_1}$ be a connected uniformly open set  containing $0.$
		Let $f:D\rightarrow \MU{R_2}$ be a uniformly real analytic free function on  which is plurisubharmonic on a neighborhood of $0$.
		Then, the function $f$ is plurisubharmonic on $D,$ and
		there are operator-valued free analytic functions $g, v^+, T$ and coanalytic $v^-$ which are defined on the whole domain $D$, such that the following are true.
 		\begin{enumerate}
			\item $T$ is contractive on $D,$
			\item $T, v^+, v^-$ vanish at $0.$
			\item $$ f(Z) = \textrm{Re }g(Z) + \bbm v^+(Z) \\ v^-(Z)\ebm^* \bbm 1 & -T(Z) \\ -T(Z)^* & 1\ebm^{-1}\bbm v^+(Z) \\ v^-(Z)\ebm,$$
			\item If 
				\beq \label{introreal} f = \textrm{Re }\hat{g}(Z) + \bbm \hat{v}^+(Z) \\ \hat{v}^-(Z)\ebm^* \bbm 1 & -\hat{T}(Z) \\ -\hat{T}(Z)^* & 1\ebm^{-1}\bbm \hat{v}^+(Z) \\ \hat{v}^-(Z)\ebm,\eeq
			where $\hat{g}, \hat{v}^+, \hat{T}$ are analytic, $v^{-}$ coanalytic such that $\hat{T},\hat{v}^+,\hat{v}^-$ vanish at $0$ then $\hat{g} = g+iC$ for some constant $C,$ and up to a change of coordinates,
			$\hat{v}^+= v^+ \oplus 0, \hat{v}^-= v^- \oplus 0, \hat{T}= T \oplus J$ for some function $J.$ That is, the representation is essentially unique up to junk terms.
		\end{enumerate}
	\end{theorem}
	By formally taking the complex Hessian of a function with such a representation, one sees that it must be plurisubharmonic wherever $T$ is contractive and all the elements are defined. That is, letting $$R(Z)= \bbm 1 & -T(Z) \\ -T(Z)^* & 1\ebm^{-1},$$
	$$w^+(Z,H) = \bbm Dv^+(Z)[H] \\ 0\ebm + R(Z)^{-1}DR(Z)[H]\bbm v^+(Z) \\ v^-(Z)\ebm,$$
	$$w^-(Z,H) = \bbm 0 \\ Dv^-(Z)[H]\ebm + R(Z)^{-1}D^*R(Z)[H]\bbm v^+(Z) \\ v^-(Z)\ebm,$$
	we have that, for functions of the form \eqref{introreal},
		$$\Delta f(Z)= w^+(Z,H)^*R(Z)w^+(Z,H) +w^-(Z,H)^*R(Z)w^-(Z,H)$$
	which is positive whenever defined and $T$ is contractive.
Theorem \ref{introrealization}
	is established as Theorem \ref{fullrepresentationformula}. Note that the realization formula in Theorem \ref{introrealization} is a realization of the form
	\eqref{butterflyrealization} composed with the analytic functions $g, v^+, T$ and $(v^-)^*.$
	For the case where $R_1$ is not finite dimensional, the notion of connectedness and real analyticity is more delicate, and when $0$ is not in the domain, our construction
	does not produce uniqueness for free. We eschew these matters, noting that in the infinite dimensional case, one could apply the direct limit construction as in \cite{royalroad},
	and in the affine case, one can probably use or adapt affine realization theories as developed in \cite{vvw12,porat}.
	
	In the case where $f$ is globally defined, it must be a sum of a hereditary plus an anti-hereditary hermitian square as in \cite{GHV,greene}. The following corollary directly from Liouville's theorem as $T(Z)$ must be constant as it is bounded.
	\begin{corollary}[Analytic Greene-Liouville theorem]
		Let $R_1$ be a finite dimensional vector space.
		Let $R_2$ be a subspace of operators.
		Let $f:\MU{R_1}\rightarrow \MU{R_2}$ be a uniformly real analytic free function which is plurisubharmonic on a neighborhood of $0$.
		Then, the function $f$ is plurisubharmonic everywhere and there are operator-valued free analytic functions $g, v^+$ and coanalytic $v^-$  defined everywhere such that 
			$$f(Z)=\textrm{Re }g(Z) + \bbm v^+(Z) \\ v^-(Z)\ebm^* \bbm v^+(Z) \\ v^-(Z)\ebm.$$
	\end{corollary}
	
	To prove Theorem \ref{introrealization}, we need a preposterous result: the monodromy theorem holds for analytic free functions on open connected free sets. 
	\begin{theorem}[Free Universal Monodromy theorem]
		The monodromy theorem holds for free analytic functions on any open connected free set $\mathcal{D}.$
	\end{theorem}
	Free Universal Monodromy is established in Section \ref{monodromysection}.
	Free Universal Monodromy implies various corollaries, such as the free inverse function theorem \cite{helkm11, PascoeMathZ, AMImp, KlepSpenko, Mancuso}
	and the universal existence of pluriharmonic conjugates. Furthermore, it shows any cohomology theory arising from sheaf theory of free analytic functions, as has been developed from differing angles \cite{AMYMan, aysym, KVVgerms},
	may be trivial on free sets.


\section{Free universal monodromy} \label{monodromysection}
	The monodromy theorem states that a complex analytic function (perhaps in several variables) on some ball $B$ analytically continues to a simply connected domain $D$ containing $B$ if and only if it analytically continues
	along each path in $D.$ Surprisingly, for analytic free functions, we see that the monodromy theorem holds on any connected free set. 
	\begin{theorem}[Free Universal Monodromy theorem]\label{monodromytheorem}
		The monodromy theorem holds for free analytic functions on any open connected free set $\mathcal{D}.$
	\end{theorem}
	We give two proofs. One is geometric and reveals why the result is true-- given two paths form one point to another, if one direct sums them and takes an $SO_2$ orbit, you get a sphere, which is simply connected.
	The second is purely in the language of free analysis and uses standard techniques, but gives little intuition as to what is ``really going on."  However, the free proof has the distinct advantage that it is totally elementary.
	\subsection{The geometric several complex variables proof}
		We now give a proof of the Free Universal Monodromy theorem based on geometry in several complex variables.
		\begin{proof}
			Suppose not, consider two non-intersecting non-self-intersecting paths $\gamma_1$ and $\gamma_2$ in $\mathcal{D}$ with the same endpoints.
			(That is, $\gamma_1(t)=\gamma_2(t)$ if and only if $t$ is equal to $0$ or $1.$)
			Consider the set
				$$\mathfrak{S} = \left\{\bpm a & b \\ -b & a\epm\bpm\gamma_1(t)& \\ & \gamma_2(t)\epm\bpm a & -b \\ b & a\epm|a^2+b^2=1, a,b\in \mathbb{R}, t \in[0,1]\right\}.$$
			A straightforward calculation gives that $\mathfrak{S}\cong S^2$ which is simply connected and thus satisfies the classical monodromy theorem from several complex variables.
		\end{proof}
	\subsection{The free proof}
		We now give a proof of the Free Universal Monodromy theorem in the style of free analysis.
		That is, we give an algebraic calculation using standard identities coming from the similarity and direct sum preserving structure of free functions
		that proves the theorem, but is perhaps conceptually deficient in terms of developing an inuition as to why. The technique is somewhat similar to that used in the
		proof of various noncommutative inverse and implicit function theorems \cite{PascoeMathZ, KlepSpenko, AMImp, Mancuso, DKVA, helkm11}.
		Later, we will use the monodromy theorem to recover inverse function theorems.
		\begin{proof}
			Suppose not, consider two paths $\gamma_1$ and $\gamma_2$ in $\mathcal{D}$ with the same endpoints.
			Note that $f$ must analytically continue along the path 
				$$\hat{\gamma}(t) = \bpm\gamma_1(t)&  \\ & \gamma_2(t)\epm.$$
			Choose $\epsilon>0$ small enough such that the same analytic continuation as for $\hat{\gamma}$ serves as an analytic continuation along the following path
				$$\gamma(t) = \bpm\gamma_1(t)& \epsilon\frac{\gamma_1(t)-\gamma_2(t)}{\|\gamma_1(t)-\gamma_2(t)\|^{1/2}} \\ & \gamma_2(t)\epm.$$
			Note that,
				$$\gamma(t)=\bpm 1 & \frac{\epsilon}{\|\gamma_1(t)-\gamma_2(t)\|^{1/2}}\\0 & 1\epm^{-1}\hat{\gamma}(t)
				\bpm 1 & \frac{\epsilon}{\|\gamma_1(t)-\gamma_2(t)\|^{1/2}}\\0 & 1\epm.$$
			Therefore,
				$$f(\gamma(t))=\bpm 1 & \frac{\epsilon}{\|\gamma_1(t)-\gamma_2(t)\|^{1/2}}\\0 & 1\epm^{-1}f(\hat{\gamma}(t))
				\bpm 1 & \frac{\epsilon}{\|\gamma_1(t)-\gamma_2(t)\|^{1/2}}\\0 & 1\epm$$
			for $t\in (0,1)$ since an analytic continuation of $f$ must preserve similarities.
			That is,
				$$f\bpm\gamma_1(t)& \epsilon\frac{\gamma_1(t)-\gamma_2(t)}{\|\gamma_1(t)-\gamma_2(t)\|^{1/2}} \\ & \gamma_2(t)\epm
			=\bpm f(\gamma_1(t))& \epsilon\frac{f(\gamma_1(t))-f(\gamma_2(t))}{\|\gamma_1(t)-\gamma_2(t)\|^{1/2}} \\ & f(\gamma_2(t))\epm.$$
			If the analytic continuation along $\gamma_1$ and $\gamma_2$ disagreed, then the analytic continuation along $\gamma$ would blow up at $1,$ so we are done.
		\end{proof}
	\subsection{Existence of harmonic conjugates}
		An immediate consequence of the Free Universal Monodromy theorem is that free pluriharmonic functions have pluriharmonic conjugates.
		We say a self-adjoint valued real free function $u$ on a set $\mathcal{D}$ is \dfn{pluriharmonic} if it is pluriharmonic at each level. That is, $\Delta u(Z)(H) \equiv 0$ for
		all $Z$ and $H.$ We say $u$ has a pluriharmonic conjugate if $u = \textrm{Re } f$ for some free analytic function $f.$
		Free pluriharmonic functions and related topics have been previously considered in \cite{GHV, greene, DHKMV,POP06, POP09, POP08, POP16,POP17,POP19,POP20multi,POP20brown}.
		\begin{corollary}[Existence of harmonic conjugates]
			 A pluriharmonic real free function defined on an open connected free set has a pluriharmonic conjugate.
		\end{corollary}
		\begin{proof}
			Let $u$ be a pluriharmonic real free function defined on an open connected free set $D.$ Note that $u$ must be real analytic.

			Note that at each point in the domain, we can write $u=\textrm{Re }f$ locally for some complex analytic function $f$.
			Moreover, any two such solutions for $f$ differ by an imaginary constant.
			It is then an elementary exercise to show that $f$ can be analytically continued along any path. 

			If $0$ were in the domain $D,$ power series methods show that $f$ can be chosen to be some specific free function on a free open set containing $0$
			and therefore, by Free Universal Monodromy, $f$ can be globally defined.
			We now discuss how to acheive this at an affine point.
			Choose a point $X_0 \in D_m.$
			We note that given $D \subseteq \MU{R_1}$ and $u:D\rightarrow \MU{R_2}$ free pluriharmonic, we can induce 
			a pluriharmonic function $u^{[m]}:D^{[m]}\rightarrow \MU{R_2\otimes M_m(\mathbb{C})}$ on $D^{[m]}=\bigcup D_{nm} \subseteq \MU{R_1\otimes M_m(\mathbb{C})}$
				by defining $u^{[m]}(Z)=u(Z).$
			Now $u^{[m]}(Z+X_0)$ has a power series about $0,$ (here $X_0$ is viewed as an element of $D^{[m]}_1$)
			and therefore is $u^{[m]}$ the real part of a free analytic function $\tilde{f}$ on all of $D^{[m]}.$
			Define the diagonal inclusion map $A: D \rightarrow D^{[m]}$ defined by 
				$$A(Z) = \bbm Z & & & \\ & Z & & \\ & & \ddots & \\ & & & Z \ebm.$$
			Define the projection map $P: \MU{R_2\otimes M_m(\mathbb{C})}\rightarrow \MU{R_2}$ to be
			$P((X_{ij})_{1 \leq i,j\leq m}) = X_{11}.$
			The function $f= P \circ \tilde{f} \circ A$ satisfies $\textrm{Re }f=u.$
		\end{proof}
		\subsubsection{The noncommutative annulus}
			Consider the set
				$$\mathbb{A}=\{X|\|X\|<2, \|X^{-1}\|<2\}.$$
			Apparently, every free pluriharmonic function on $\mathbb{A}$ has a pluriharmonic conjugate.
			However, the set $\mathbb{A}$ appears on face to be an analogue of the annulus, which in classical complex analysis famously has functions defined on
			it which do not have harmonic conjugates,
			such as $\log |z|$, as the annulus is not simply connected. Therefore, the Free Universal Monodromy theorem could be considered somewhat
			disturbing. The resolution is obviously that $\log |X|$ is not pluriharmonic on $\mathbb{A}.$
			The following corollary characterizes all free pluriharmonic functions on $\mathbb{A},$ which follows from the fact that they
			must be the real part of an analytic function, which must itself have a Laurent series.
			\begin{corollary}[Pluriharmonic free functions on the noncommutative annulus arise from Laurent series]
			 Let 
				$$\mathbb{A}=\{X|\|X\|<2, \|X^{-1}\|<2\}.$$
			Every pluriharmonic free function on $\mathbb{A}$ is of the form
				$$\textrm{Re }\sum^{\infty}_{n=-\infty} c_nX^n.$$
			\end{corollary}
	\subsection{The local-global inverse function theorem}
		The inverse function theorem in free analysis is surprisingly strong.
 A function whose derivative is pointwise bijective
		is bijective onto its range \cite{PascoeMathZ, KlepSpenko, AMImp, Mancuso, DKVA, helkm11, augat}.
Furthermore, in the free case, the Jacobian conjecture holds: a noncommuative polynomial map with pointwise nonsingular derivative is invertible and has a polynomial inverse \cite{PascoeMathZ, augat}, the proof of which relies heavily on the strength of the free inverse function theorem.

 We now show how the free inverse function theorem follows from Free Universal Monodromy.
		We consider this particularly appealing because it gives a geometric explanation for something that algebraically looks like a coincidence. We should note, however, that the original version
		in \cite{PascoeMathZ} does not require analyticity assumptions and uses some algebraic version of the derivative. (Perhaps, this suggests there is an
		algebraic analogue of Free Universal Monodromy that holds in vaster generality.)
		\begin{theorem}[Pascoe \cite{PascoeMathZ}]\label{pascoemathz}
			Let $U \subseteq \MU{\mathbb{C}^d}$ be an open connected free set.
			Let $F: U \rightarrow \MU{\mathbb{C}^d}$ be an analytic free function.
			The following are equivalent:
			\begin{enumerate}
				\item $DF(X)$ is always nonsingular,
				\item $F^{-1}$ exists and is an analytic free function.
 			\end{enumerate}
		\end{theorem}
		\begin{proof}
			For each point in the range, $F^{-1}$ can be locally defined by the classical inverse function theorem. Therefore, $F^{-1}$ can be globally defined on the range by Free Universal Monodromy.
		\end{proof}
	\subsection{The logarithm}
		Classical monodromy theory gives insight into when one can take the logarithm or square root of a function. In the free case, the utility is somewhat limited.
		In terms of the inverse and implicit function theorems, the derivative of the implicit defining function is not always full rank.
		We now give an example of obstructions to defining the logarithm of a nonsingular function which is related to questions of convergence of the
		Baker-Campbell-Hausdorff formula \cite{CBHobstructions} which we discuss at the end of this subsection.
		\subsubsection{Logarithms of analytic functions on simply connected domains taking values in $GL_n$ may not exist}
			Let $U\subseteq \mathbb{C}$ simply connected. Let $f: U \rightarrow \mathbb{C}\setminus \{0\}$ be analytic. It is a well known fact that there
			is a function $g:U \rightarrow \mathbb{C}$ such that $f=e^g.$ We consider the corresponding question for matrix valued functions $f: U \rightarrow GL_n(\mathbb{C}).$
			We will give an example of a polynomial function $f: \mathbb{D} \rightarrow GL_n(\mathbb{C})$
			 such that $f(0)=1$ such that the principal branch of $\log f$ does not analytically continue to the whole of $\mathbb{C}$.

			\begin{theorem}
			Define 
			$$f(z)=\bbm 1 & z \\ z & 1+ z^2 \ebm.$$
			We consider the function 
			$$\log f = \sum \frac{(-1)^{n+1}}{n}\bbm 0& z \\ z & z^2 \ebm^n = \sum c_n z^n.$$
			The radius of convergence for the series $\sum c_nz^n$ for $\log f$ is $2$ and
			therefore the largest disk around $0$ it can analytically continue to is $2\mathbb{D}.$
			\end{theorem}
			\begin{proof}
				First, note that, near $0,$ and thus by analytic continuation everywhere,
					$$\textrm{Tr } \log f(z) = \log \det f(z) \equiv 0.$$ 
				Second, note that $f(z)$ has repeated eigenvalues only at $-2i, 2i,$ and $0.$ Moreover, $f(z)$ is not diagonalizable at $\pm 2i$ and the repeated
				eigenvalue is $-1.$
				Thirdly, note that if $M$ is a $2$ by $2$
				nondiagonalizable matrix and $e^A=M,$ then $A$ must be nondiagonalizable and thus have repeated eigenvalues.
				Suppose $f(z)$ analytically continued along the path from $0$ to $\pm 2i,$ then $f(\pm 2i)$ must have repeated eigenvalues whose exponential is
				$-1$ and therefore the trace of $\log f(\pm 2i)$ is nonzero, which would be a contradiction.
				
				Note that away from the points where $f(z)$ has repeated eigenvalues, the eigenvalues are given by analytic functions and therefore $\log f$ continues
				to any
				simply connected set containing $0$ and avoiding $\pm 2i.$
			\end{proof}

			Should we be surprised the radius is exactly $2?$ Evidently not. Again, consider the eigenvalues of
			$$\bbm 1 & z \\ z & 1+ z^2 \ebm.$$
			As we approach to $2i,$ the two eigenvalues approach $-1$, but one approaches from a clockwise
			direction while the other takes a counterclockwise approach, which one see
			by looking at the determinant and the trace of $f(z)$ along the path from $0$ to $2i$.
			If we upper triangularize $f(z)$ along this path, we have that
					$$f(z) = u(z)^*\bbm\lambda_1(z) & d(z) \\ 0 & \lambda_2(z)\ebm u(z)$$
			where $u(z)$ is unitary and $d(z)$ is some function which is non-zero at $2i.$
			If we assume (for heuristic purposes only) that $u(z)$ and $d(z)$ were chosen real analytically, we have that
			$$\log f(z) = u(z)^*\bbm \log \lambda_1(z) & d(z)\frac{\log\lambda_1(z)-\log\lambda_2(z)}{\lambda_1(z)-\lambda_2(z)} \\ 0 & \log \lambda_2(z)\ebm u(z)$$
			which blows up as the branches of $\log$ do not agree.
			That is, the path from $2i$ to $-2i$ witnesses the path along which the function cannot continue, and therefore monodromy does not apply.
			A similar argument shows that $\sqrt{f(z)}$ cannot be defined on all of $\mathbb{D}.$

		\subsubsection{The Martin-Shamovich logarithm problem and relation to the Baker-Campbell-Hausdorff formula}
			Martin and Shamovich posed the problem of whether or not an element of the
			Fock space which takes invertible values on the row ball has a well-defined logarithm on the row ball.
			We now adapt our example to show this is impossible.

			Consider $F(X,Y) = e^{X}e^{Y}.$ Note $F(X,Y)$ is invertible for any inputs. Substitute in the tuple
			$$X= \bbm 0 & 0 \\ z & 0 \ebm, Y = \bbm 0 & z \\ 0 & 0 \ebm.$$
			Now $F(X,Y)=f(z)$ and therefore has $\log F$ has finite radius of convergence.

			The quantity $\log e^Xe^Y$ is of importance in Lie theory, the power series for it at $0$ is called the Baker-Campbell-Hausdorff formula.
			The series for $\log e^Xe^Y$ is known to converge whenever $\|X\|+\|Y\| \leq \frac{\log 2}{2}$ \cite{CBHobstructions}.
			The fact that the convergence is not global was used in our counterexample to the Martin-Shamovich logarithm problem.
			Our Free Universal Monodromy theorem gives a geometric explanation, there are paths along which the
			function does not continue. 
	\subsection{Sheaf theory}
		In free noncommutative function theory, there have been at least two recent approaches to the development of sheaf theory. Agler, McCarthy and Young, in the process of investigating
		noncommutative symmetric functions in two variables, developed an intricate theory of noncommutative manifolds \cite{AMYMan, aysym}. On the other hand, Klep, Vinnikov and Volcic took an approach
		from germ theory \cite{KVVgerms}. The Free Universal Monodromy theorem says, in the case of free sets, that the theory arising from analytic continuation is essentially trivial. That is, in the global case
		there may be no nontrivial cohomology type theory.

\section{Plurisubharmonic functions}
	\subsection{The local realization}
		The goal of this subsection is to establish a local version of Theorem \ref{introrealization}. We discuss globalization in a later subsection.
		\subsubsection{The middle matrix}
			The first step is use the positivity of the complex hessian of plurisubharmonic free function to show that some matrix assembled from power series coefficients is positive semidefinite and then do some kind of 				Gelfand-Naimark-Segal construction to obtain a realization. Essentially, this adapts the middle matrix type constructions \cite{helmconvex04,GHV,greene} to an infinite case and
			then uses them to generate a representation formula, as opposed to the previously used application,
			which was to Sherlockianly rule out possible forms until all that was left was a small class of polynomials.

			The following technical lemma, in the vein of the Camino-Helton-Skelton-Ye lemma \cite{CHSY}, will be needed for further discourse. Usually these are proven for a finite list of monomials,
			or something like that, and are ubiquitous. We give a generalization for an infinite sequence of linearly independent real free functions. Another advantage is that
the proof of the original Camino-Helton-Skelton-Ye lemma spans several pages, whereas our method is somewhat concise. (However, we  do not obtain explicit size bounds in the finite dimensional case.)
			\begin{lemma}\label{extracheesy}
				Let $h= (h_i)_i$ be a vector valued real free function which is bounded on a neighborhood of $0$ such that $h_i$ are $\ell^2$-linearly independent.
				Let $$C=(c_{i,j})_{i,j}$$
				be a infinite block operator matrix.
				If $$\sum_{i,j} c_{i,j}h_i^*vv^*h_j\geq 0$$
				on a neighborhood of $0$ and for all vectors $v.$
				then $C$ is positive semidefinite.
			\end{lemma}
			\begin{proof}
				Consider the vector space,
					$$V_{X,v} = \textrm{Range }v^*h(X).$$
				Note, $$V_{X_1,v_1} + V_{X_2,v_2} = V_{X_1\oplus X_2, v_1\oplus v_2}$$
				Let $V = \bigcup_{X,v} V_{X,v}.$
				Note $V$ is a vector space and $C$ is positive on $\overline{V}.$
				
				Just suppose $(w_i)^{\infty}_{i=1}=w \perp V.$
				Then, $\sum \overline{w_i}v^*h_i(X) = 0$ for all $X, v.$ Therefore $h_i$ are $\ell^2$ linearly dependent.

				Therefore, $\overline{V}$ is the whole Hilbert space and thus $C$ is positive.
			\end{proof}
			We now show that two structured infinite block matrices coming from coefficients,
			$C^+=[c_{\alpha^* z_i^*z_j \beta}]_{z_i\alpha,z_j\beta}, C^-=[c_{\alpha^* z_iz_j^* \beta}]_{z_i^*\alpha,z_j^*\beta}$ are positive semidefinite whenever a function $f(Z) = \sum_{\alpha} c_{\alpha}Z^{\alpha}$ is plurisubharmonic near $0.$
			\begin{lemma}\label{middlematrix}
				Let $f(Z) = \sum_{\alpha} c_{\alpha}Z^{\alpha}$ be a noncommutative power series with operator coefficients which is convergent on some uniform neighborhood of $0.$
				If $f$ is plurisubharmonic on a neighborhood of $0$, then the infinite block matrices 
					$$C^+=[c_{\alpha^* z_i^*z_j \beta}]_{z_i\alpha,z_j\beta}, C^-=[c_{\alpha^* z_iz_j^* \beta}]_{z_i^*\alpha,z_j^*\beta}$$
				are positive semidefinite.
			\end{lemma}
			\begin{proof}
				The complex Hessian is given by the formula
					$$\Delta f(Z)[H] = \sum c_{\alpha^*z_i^*\gamma z_j\beta}Z^{\alpha^*}H_i^*Z^\gamma H_jZ^{\beta}
					+ c_{\alpha^*z_i\gamma z_j^*\beta}Z^{\alpha^*}H_iZ^\gamma H_j^*Z^{\beta} \geq 0.$$
				Evalute at
					$$Z=\bpm Z & \\ &0 \epm, H=\bpm  & 0  \\ v^*Z & \epm$$
				and take the block $1-1$ entry
				to get that
					$$\sum c_{\alpha^*z_i^*z_j\beta}Z^{\alpha^*}Z_i^*vv^*Z_jZ^{\beta} \geq 0.$$
				As the functions $Z_jZ^{\beta}$ are linearly independent, we see by Lemma \ref{extracheesy} that 
					$$C^+=[c_{\alpha^* z_i^*z_j \beta}]_{z_j\beta,z_i\alpha}\geq 0.$$
				The case for $C^-$ is similar.
			\end{proof}
		\subsubsection{The construction}
			We now proceed to give the construction for the local realization formula for a plurisubharmonic free function. Mostly, the statement of the lemma gives the proof,
			but a few things like boundedness of the elements involved must be nontrivially checked.
			\begin{lemma}\label{representationformula}
				Let $f(Z) = \sum_{\alpha} c_{\alpha}Z^{\alpha}$ be a noncommutative power series with operator coefficients which is convergent on some uniform neighborhood of $0.$
				Let $f$ be plurisubharmonic on a neighborhood of $0.$
				By Lemma \ref{middlematrix}, the infinite block matrices
					$$C^+=[c_{\alpha^* z_i^*z_j \beta}]_{z_i\alpha,z_j\beta}, C^-=[c_{\alpha^* z_iz_j^* \beta}]_{z_i^*\alpha,z_j^*\beta}$$
				are positive semidefinite.
				Let $\hilbert^+$ and $\hilbert^-$ denote the corresponding Hilbert spaces the inner products arising from $C^+$ and $C^-$ respectively. Denote the vector corresponding to the word $\alpha\otimes w$ by $\vex{\alpha\otimes w}.$
				For analytic words $\alpha,$ define the map $T_\alpha: \hilbert^- \rightarrow \hilbert^+$ by $$T_{\alpha} \vex{\beta\otimes w} = \vex{\alpha\beta \otimes w}.$$
				Define $$T(Z) = \sum_{\alpha \textrm{ analytic}, |\alpha|>0} T_\alpha Z^{\alpha}.$$
				Define $$v^+(Z) = \sum_{\alpha \textrm{ analytic}, |\alpha|>0} Q_{\alpha}Z^{\alpha}, v^{-}(Z) = \sum_{\alpha \textrm{ analytic}, |\alpha|>0} Q_{\alpha^*}Z^{\alpha^*}$$
				where $Q_\alpha w = \vex{\alpha \otimes w}.$
				Define $$g(Z) = c_0 + 2\sum_{\alpha \textrm{ analytic}, |\alpha|>0} c_{\alpha}Z^{\alpha}.$$
				The following are true:
				\begin{enumerate}
					\item $T_{\alpha}^*\vex{\beta\otimes w}=\vex{\alpha^*\beta\otimes w},$
					\item $\sup_{\alpha} \|T_{\alpha}\|^{\frac{1}{|\alpha|}}<\infty$,
					\item $T, v^+, v^-,g$ are defined on a uniformly open neighborhood of $0,$ 
					\item $T$ is contractive on a uniformly open neighborhood of $0,$
					\item $$ f(Z) = \textrm{Re }g(Z) + \bbm v^+(Z) \\ v^-(Z)\ebm^* \bbm 1 & -T(Z) \\ -T(Z)^* & 1\ebm^{-1}\bbm v^+(Z) \\ v^-(Z)\ebm.$$
					\item If 
				$$ f = \textrm{Re }\hat{g}(Z) + \bbm \hat{v}^+(Z) \\ \hat{v}^-(Z)\ebm^* \bbm 1 & -\hat{T}(Z) \\ -\hat{T}(Z)^* & 1\ebm^{-1}\bbm \hat{v}^+(Z) \\ \hat{v}^-(Z)\ebm,$$
			where $\hat{g}, \hat{v}^+, \hat{T}$ are analytic, $v^{-}$ coanalytic such that $\hat{T},\hat{v}^+,\hat{v}^-$ vanish at $0$ then $\hat{g} = g+iC$ for some constant $C,$ and up to a change of coordinates,
			$\hat{v}^+= v^+ \oplus 0, \hat{v}^-= v^- \oplus 0, \hat{T}= T \oplus J$ for some function $J.$ That is, the representation is essentially unique up to junk terms.
				\end{enumerate}
			\end{lemma}
			\begin{proof}
				The theorem is scale invariant, so we without loss of generality assume all the power series coefficients $c_{\alpha}$ are contractive.
				
				(1) $$\langle T_{\alpha}^*\vex{\beta\otimes w},\vex{\gamma\otimes u}\rangle = \langle \vex{\beta\otimes w},T_{\alpha}\vex{\gamma\otimes u}\rangle 
				= \langle c_{\gamma^*\alpha^*\beta}w,u\rangle = \langle \vex{\alpha^*\beta \otimes w},\vex{\gamma\otimes u} \rangle.$$

				(2) Note that $$\|T_{\alpha}\|^2 = \|T_{\alpha}^*T_{\alpha}\|=\sqrt{\rho(T_{\alpha}^*T_{\alpha})}.$$
				Now, note 
					$$\rho(T_{\alpha}^*T_{\alpha}) = \sup_{\beta, w} \langle(T_{\alpha}^*T_{\alpha})^n\vex{\beta\otimes w}, \vex{\beta\otimes w}\rangle^{1/n} =
					\sup_{\beta, w}\langle c_{\beta^*\alpha^*\alpha\beta}w,w\rangle^{1/n}$$
				which is uniformly bounded as the radius of convergence of the series is positive.

				(3) Follows from the fact that the $T_\alpha, Q_\alpha$ are uniformly bounded. Note $\|Q_\alpha w\|^2=\langle c_{\alpha^*\alpha}w,w\rangle \leq \|w\|^2.$

				(4) Follows from the fact that $T$ is uniformly analytic and defined on a neighborhood of $0.$

				(5) This is elementary generatingfunctionology.
				Specifically expanding out the terms in the formal power series for the resolvent
					$$\bbm \hat{v}^+(Z) \\ \hat{v}^-(Z)\ebm^* \bbm 1 & -\hat{T}(Z) \\ -\hat{T}(Z)^* & 1\ebm^{-1}\bbm \hat{v}^+(Z) \\ \hat{v}^-(Z)\ebm$$
				we see that the coefficient of $Z^{\beta}$ is given by $$Q^*_{\alpha_0^*}T_{\alpha_1}\ldots T_{\alpha_{n-1}} Q_{\alpha_n} = c_\beta$$
				where  $\beta = \alpha_0\alpha_1\ldots\alpha_{n-1}\alpha_n$ and the $\alpha_i$ alternate between being analytic and coanalytic
				and $T_{\alpha^*}$ is formally defined to be $T_{\alpha}^*$ whenever $\alpha$ is a nontrivial analytic word.

				(6) Expanding the resolvent and equating term by term says there is an isometry witnessing this decomposition.
				Write $\hat{T}(Z) = \sum \hat{T}_\alpha Z^{\alpha},$ $\hat{v}^+ = \sum \hat{Q}_\alpha Z^\alpha,$
				$\hat{v}^- = \sum \hat{Q}_{\alpha^*} Z^{\alpha^*}.$
				Denote the Hilbert spaces corresponding to the decomposition of the resolvent as $\hat{\hilbert^+}$ and $\hat{\hilbert^-}$
				Define $\hat{\beta}_w = \hat{T}_{\alpha_1}\ldots \hat{T}_{\alpha_{n-1}} \hat{Q}_{\alpha_n}w$
				where $\beta = \alpha_1\ldots\alpha_{n-1}\alpha_n$ and the $\alpha_i$ alternate between being analytic and coanalytic
				and $\hat{T}_{\alpha^*}$ is formally defined to be $\hat{T}_{\alpha}^*$ whenever $\alpha$ is a nontrivial analytic word.
				Now note, when defined,
					$$\langle \hat{\beta}_w,\hat{\eta}_\omega \rangle = \langle \vex{\beta\otimes w},\vex{\eta\otimes \omega} \rangle = 
					\omega^*c_{\eta^*\beta}w.$$
				Therefore, there are isometries $V^+: \hilbert^+ \rightarrow \hat{\hilbert^+}$ and 
				$V^-: \hilbert^- \rightarrow \hat{\hilbert^-}$ such that
					$$\hat{Q}_\alpha = V^+Q_\alpha, \hat{Q}_{\alpha^*} = V^-Q_{\alpha^*}, \hat{T}_\alpha V^+ = V^{-}T_{\alpha},
					\hat{T}_{\alpha}^* V^- = V^{+}T_{\alpha}^*.$$
				Without loss of generality, assume $$\hilbert^+ \subseteq \hat{\hilbert^+}$$ and $$\hilbert^- \subseteq \hat{\hilbert^-}$$
				and $V^+, V^-$ are isometric inclusion maps. We now decompose $\hat{v^+}, \hat{v^-}, \hat{T}$ along the natural decompositions
				$\hat{\hilbert^+} = \hilbert^+ \oplus (\hilbert^+)^{\perp},$ and $\hat{\hilbert^-} = \hilbert^- \oplus (\hilbert^-)^{\perp}.$

				Now, $$\hat{v^+} = V^+v^+ = \bbm 1 \\ 0\ebm v^+ = \bbm v^+ \\ 0 \ebm,$$
				     $$\hat{v^-} = V^-v^- = \bbm 1 \\ 0\ebm v^- = \bbm v^- \\ 0 \ebm.$$
				Write
					$$\hat{T_{\alpha}}= \bbm \hat{T_{\alpha}}^{11} & \hat{T_{\alpha}}^{12} \\ \hat{T_{\alpha}}^{21} & \hat{T_{\alpha}}^{22} \ebm.$$
				Now,
					$$\bbm \hat{T_{\alpha}}^{11} \\ \hat{T_{\alpha}}^{21} \ebm =\hat{T_{\alpha}}V^+= V^-T_{\alpha} = 
					\bbm T_{\alpha} \\ 0 \ebm,$$
				and therefore  $\hat{T_{\alpha}}^{11} = T_{\alpha},$ $\hat{T_{\alpha}}^{21}=0.$
				The relation for $T_{\alpha}^*$ similarly implies that $\hat{T_{\alpha}}^{12}=0$ and therefore $\hat{T} = T \oplus J.$
					
			\end{proof}
	\subsection{Plurisubharmonicity is geometric}
			We have the following corollary of Lemma \ref{representationformula}.
			\begin{corollary}
				Let $f$ be a uniformly real analytic free function on a connected uniformly open set $D$ which is plurisubharmonic 
				on some perhaps smaller open free set $B$. Then, $f$ is plurisubharmonic on $D.$
			\end{corollary}
			\begin{proof}
				Tracing through the proof of Lemma \ref{representationformula}, we see that the domain of convegence of $v^+, v^-, T, g$
				depend only only on the radius of convegence of the original power series. (Not on the radius of plurisubharmonicity, which could \emph{a priori} be smaller.) Moreover, the radius such that $T$ is contractive 
				also depends only on this radius of convegence. That is, it is independent of the assumed domain where it was plurisubharmonic.
				Note that $f$ must be plurisubharmonic wherever $T$ is contractive by taking the complex Hessian. By applying this observation
				along paths, we see that $f$ is plurisubharmonic everywhere. 
			\end{proof}
	\subsection{The global realization}
		\subsubsection{Some algebraic identities and inequalities}
			We now collect some algebraic identities and inequalities that will be needed to affinize realizations.
			\begin{lemma}\label{posreal}
				Suppose $1-A-A^*$ is positive definite, then $1-A$ is invertible and moreover $A(1-A)^{-1}$ is a strict contraction.
			\end{lemma}
			\begin{proof}
				We will show that the real part of $1-A$ is positive definite, and hence $1-A$ must be invertible.
				Note that $2\textrm{Re }(1-A) = 2-A-A^* \geq 1-A-A^*.$ 
				Note $1-(1-A^*)^{-1}A^*A(1-A)^{-1}$ is positive whenever $(1-A^*)(1-A)-A^*A=1-A-A^*$ is positive, and therefore
				$A(1-A)^{-1}$ is a strict contraction.
			\end{proof}
			The following is an algebraic fact.
			\begin{lemma}\label{restructurefine}
				Suppose $1-A-A^*$ is positive definite.
				The following are equivalent:
				\begin{enumerate}
					\item $D^*(1-A-A^*)^{-1}C$
					\item $\bbms (1-A)^{-1}D \\ 0 \ebms^* \bbms 1 & -A(1-A)^{-1} \\  -A^*(1-A^*)^{-1} &1 \ebms^{-1} \bbms (1-A)^{-1}C \\ 0 \ebms$
					\item $\bbms 0 \\ (1-A^*)^{-1}D \ebms^* \bbms 1 & -A(1-A)^{-1} \\  -A^*(1-A^*)^{-1} &1 \ebms^{-1} \bbms (1-A)^{-1}C \\ 0 \ebms+D^*(1-A)^{-1}C$
					\item $\bbms (1-A)^{-1}D \\ 0 \ebms^* \bbms 1 & -A(1-A)^{-1} \\  -A^*(1-A^*)^{-1} &1 \ebms^{-1} \bbms 0 \\ (1-A^*)^{-1}C \ebms +D^*(1-A^*)^{-1}C $
					\item $\bbms 0 \\ (1-A^*)^{-1}D \ebms^* \bbms 1 & -A(1-A)^{-1} \\  -A^*(1-A^*)^{-1} &1 \ebms^{-1} \bbms 0 \\ (1-A^*)^{-1}C \ebms $
					\item $\bbms A(1-A)^{-1}D \\ A^*(1-A^*)^{-1}D \ebms^*
					\bbms 1 & -A(1-A)^{-1} \\  -A^*(1-A^*)^{-1} &1 \ebms^{-1} \bbms (1-A)^{-1}C \\ 0 \ebms + D^*(1-A)^{-1}C $
					\item $\bbms A(1-A)^{-1}D \\ A^*(1-A^*)^{-1}D \ebms^*
					\bbms 1 & -A(1-A)^{-1} \\  -A^*(1-A^*)^{-1} &1 \ebms^{-1} \bbms 0 \\ (1-A^*)^{-1}C \ebms + D^*(1-A^*)^{-1}C $
					\item $\bbms (1-A)^{-1}D \\ 0 \ebms^*
					\bbms 1 & -A(1-A)^{-1} \\  -A^*(1-A^*)^{-1} &1 \ebms^{-1} \bbms A(1-A)^{-1}C \\ A^*(1-A^*)^{-1}C \ebms + D^*(1-A^*)^{-1}C $
					\item $\bbms 0 \\ (1-A^*)^{-1}D \ebms^*
					\bbms 1 & -A(1-A)^{-1} \\  -A^*(1-A^*)^{-1} &1 \ebms^{-1} \bbms A(1-A)^{-1}C \\ A^*(1-A^*)^{-1}C \ebms + D^*(1-A)^{-1}C $
					\item $\bbms A(1-A)^{-1}D \\ A^*(1-A^*)^{-1}D \ebms^*
					\bbms 1 & -A(1-A)^{-1} \\  -A^*(1-A^*)^{-1} &1 \ebms^{-1} \bbms A(1-A)^{-1}C \\ A^*(1-A^*)^{-1}C \ebms  +D^*[(1-A)^{-1}+(1-A^*)^{-1}-1]C $
				\end{enumerate}
			\end{lemma}
			\begin{proof}
				We can interpret the formulas as follows. We start with $C$ multiplied by some number of powers of $A$ (or $A^*$,) then multiply
				by some nonzero power of $A^*$ and then some nonzero power of $A$ and
				so on until we choose to stop and multiply by some power of $A^*$ (or $A$) and then finally
				$D^*.$ The formulas make sense because of Lemma \ref{posreal}. The interested reader may be interested in formally checking
				the identities by hand or using computer algebra software such as NCAlgebra in Mathematica \cite{ncalgebra}.

				We will concretely prove (1) equals (2) by hand for demonstration purposes.
				Consider
					$$\bbms (1-A)^{-1}D \\ 0 \ebms^* \bbms 1 & A(1-A)^{-1} \\  A^*(1-A^*)^{-1} &1 \ebms^{-1} \bbms (1-A)^{-1}C \\ 0 \ebms.$$
				Note, using the formula for the inverse of a block two by two matrix,
					 $$\bbms 1 & A(1-A)^{-1} \\  A^*(1-A^*)^{-1} &1 \ebms^{-1}=\bbms (1-A(1-A)^{-1}A^*(1-A^*)^{-1})^{-1} & * \\  * &* \ebms^{-1}.$$
				Therefore, the desired quanitity is equal to
					$$((1-A)^{-1}D)^*(1-A(1-A)^{-1}A^*(1-A^*)^{-1})^{-1}(1-A)^{-1}C.$$
				Bringing the inverses inside gives
					$$D^*((1-A)(1-A^*)-AA^*)^{-1}C=D^*(1-A-A^*)^{-1}C.$$
			\end{proof}
		\subsubsection{Pushing the realization around}
			The following is an algebraic fact.
			\begin{lemma}\label{restructure}
				Let $v^+, T$ be analytic free functions and $v^-$ be coanalytic
				on some uniformly open neighborhood of $0$ which vanish at $0,$ and let $v_0$ be a constant operator.
				Let
					$$f(Z) = (v^+(Z) + v^{-}(Z)+v_0)^*(1-T(Z)-T(Z)^*)^{-1}(v^+(Z) + v^{-}(Z)+v_0).$$
				Consider the functions 
				$$\hat{v}^+(Z)=(1-T(Z))^{-1}(v^+(Z)+T(Z)v_0),$$
				$$\hat{v}^-(Z)=(1-T(Z)^*)^{-1}(v^-(Z)+T(Z)^*v_0),$$
				$$\hat{T}(Z)=T(Z)(1-T(Z))^{-1},$$
				$$\hat{g}(z)=2(v^-(Z)+v_0)^*(1-T(Z))^{-1}(v^+(Z)+v_0) - v_0^*v_0.$$
				Each of the above are well defined whenever $1-T(Z)-T(Z)^*$ was positive, and $\hat{T}(Z)$ is contractive there.
				Moreover,
				$$f(Z) = \textrm{Re }\hat{g}(Z) + \bbm \hat{v}^+(Z) \\ \hat{v}^-(Z)\ebm^* \bbm 1 & -\hat{T}(Z) \\ -\hat{T}(Z)^* & 1\ebm^{-1}\bbm \hat{v}^+(Z) \\ \hat{v}^-(Z)\ebm.$$
			\end{lemma}
			\begin{proof}
				We distribute the product into 9 terms apply Lemma \ref{restructurefine} nine times. Specifically, one applies relation $2$ to $D=v^+, C=v^+,$
				relation $3$ to $D=v^-, C=v^+,$ relation $4$ to $D=v^+, C=v^-,$ relation $5$ to $D=v^-, C=v^-,$
				relation $6$ to $D=v_0, C=v^+,$ relation $7$ to $D=v_0, C=v^-,$
				relation $8$ to $D=v^+, C=v_0,$ relation $9$ to $D=v^-, C=v_0,$
				and finally relation $10$ to $D=v_0, C=v_0$ and collecting terms. That $\hat{T}$ is contractive follows from Lemma \ref{posreal}.
			\end{proof}
			The following lemma is another algebraic fact.
			\begin{lemma}\label{movecenter}
				Let $v^+, T$ be analytic free functions and $v^-$ coanalytic on some uniformly open neighborhood of $0$ which vanish at $0.$
				Consider
					$$f(Z) = (v^+(Z) + v^{-}(Z))^*(1-T(Z)-T(Z)^*)^{-1}(v^+(Z) + v^{-}(Z)).$$
				Suppose $W$ is in the common domain of the components of the realization and $1-T(W)-T(W)^*$ is positive.
				Then,
					$$f(Z+W)= (\hat{v}^+(Z) + \hat{v}^{-}(Z)+\hat{v_0})^*(1-\hat{T}(Z)-\hat{T}(Z)^*)^{-1}(\hat{v}^+(Z) + \hat{v}^{-}(Z)+\hat{v_0})$$
				where 
					$u = \sqrt{1-T(W)-T(W)^*},
					\hat{T}(Z)= u^{-1}(T(Z)-T(W))u^{-1},
					\hat{v}^+(Z) = u^{-1}(v^{+}(Z)-v^+(W)),
					\hat{v}^-(Z) = u^{-1}(v^{-}(Z)-v^-(W)),
					\hat{v_0} = u^{-1}(v^+(W)+v^-(W)),$
				which are all defined on the neighborhood where the original quantities were defined, translated by $W.$
			\end{lemma}
			\begin{proof}
				The proof is left to the reader. (That is, one essentially just evaluates the formula and simplifies.)
			\end{proof}
			Now we see that the realization for $f$ must analytically continue due to canonicity of the construction.
			\begin{lemma}\label{smallcontinuation}
				Let $f$ be a uniformly real analytic plurisubharmonic free function on a connected domain $D$ in $d$ variables containing $0$.
				Let $v^{-},v^{+}, T, g$ be as in Lemma \ref{representationformula}.
				Let $W$ be in their common domain, such that the components of the realization at $W$ are also defined at $0.$
				Let $v_W^{-},v_W^{+}, T_W, g_W$ be as in Lemma \ref{representationformula} for $f(Z+W).$
				Then, $v^+, v^-, T, g$ analytically continue to the common domain of $v_W^{-},v_W^{+}, T_W, g_W.$
			\end{lemma}
			\begin{proof}
				Compute the realization at $W.$ Using Lemma \ref{movecenter} and then Lemma \ref{restructure} we obtain a realization at $0$ as in Lemma \ref{representationformula} which is defined 
				on the domain of the realization at $W$,
				which by (6) of Lemma \ref{representationformula} has a domain which is smaller than that of the realization computed directly at $0$ and so therefore the realization analytically continues
 				to the appropriate domain. 
			\end{proof}
			Finally, canonicity plays well with computation of affine realizations.
			\begin{observation} \label{resolvent}
				We note that given $D \subseteq \MU{R_1}$ and $f:D\rightarrow \MU{R_2}$ free plurisubharmonic, we can induce 
				a plurisubharmonic function $f^{[m]}:D^{[m]}\rightarrow \MU{R_2\otimes M_m(\mathbb{C})}$ on $D^{[m]}=\bigcup D_{nm} \subseteq \MU{R_1\otimes M_m(\mathbb{C})}$
				by defining $f^{[m]}(Z)=f(Z).$
				In the construction in Lemma \ref{representationformula}, one can show that
				$T^{[m]}(Z) = T(Z), (v^+)^{[m]}(Z)=v^+(Z), (v^-)^{[m]}(Z)=v^-(Z), g^{[m]}(Z)=g(Z).$ This follows directly from the Hankel matrix type construction,
				although computing the exact form of the corresponding $C^+,$ $C^-$ does not appear to be particularly useful.
			\end{observation}
		\subsubsection{Globalization}
			Now, as a direct consequence of Lemma \ref{smallcontinuation} combined with Observation \ref{resolvent} and Free Universal Monodromy, we have the main result.
			\begin{theorem}\label{fullrepresentationformula}
				Let $R_1$ be a finite dimensional vector space.
				Let $R_2$ be a subspace of operators.
				Let $D \subseteq \MU{R_1}$ be a connected uniformly open set  containing $0.$
				Let $f:D\rightarrow \MU{R_2}$ be a uniformly real analytic free function on  which is plurisubharmonic on a neighborhood of $0$.
				Then, the function $f$ is plurisubharmonic on $D,$ and
				there are operator-valued free analytic functions $g, v^+, T$ and coanalytic $v^-$ which are defined on the whole domain $D$, such that the following are true.
		 		\begin{enumerate}
					\item $T$ is contractive on $D,$
					\item $T, v^+, v^-$ vanish at $0.$
					\item $$ f(Z) = \textrm{Re }g(Z) + \bbm v^+(Z) \\ v^-(Z)\ebm^* \bbm 1 & -T(Z) \\ -T(Z)^* & 1\ebm^{-1}\bbm v^+(Z) \\ v^-(Z)\ebm,$$
					\item If 
				$$ f = \textrm{Re }\hat{g}(Z) + \bbm \hat{v}^+(Z) \\ \hat{v}^-(Z)\ebm^* \bbm 1 & -\hat{T}(Z) \\ -\hat{T}(Z)^* & 1\ebm^{-1}\bbm \hat{v}^+(Z) \\ \hat{v}^-(Z)\ebm,$$
			where $\hat{g}, \hat{v}^+, \hat{T}$ are analytic, $v^{-}$ coanalytic such that $\hat{T},\hat{v}^+,\hat{v}^-$ vanish at $0$ then $\hat{g} = g+iC$ for some constant $C,$ and up to a change of coordinates,
			$\hat{v}^+= v^+ \oplus 0, \hat{v}^-= v^- \oplus 0, \hat{T}= T \oplus J$ for some function $J.$ That is, the representation is essentially unique up to junk terms.
				\end{enumerate}
			\end{theorem}

\section{Questions}
	We close with a series of questions, motivated by our investigation of Free Universal Monodromy and plurisubharmonicity, with vary levels of well-definedness as to what constitutes an ``answer."
	\begin{enumerate}
		\item Does Free Universal Monodromy extend to algebraic contexts as was the case for free inverse function theorems in \cite{PascoeMathZ}? 
		\item When does a nonsingular free noncommutative function possess a logarithm or a square root?
		\item What is the theory of partial differential equations in free noncommutative function theory? Free Universal Monodromy suggests that 
		the existence-uniqueness theory should be robust. 
		\item For both monotonicity and convexity in \cite{royalroad} and plurisubharmonicity here,
		the positivity of certain Hankel matrices of power series coefficients was locally translation invariant. (See Lemma \ref{middlematrix}.)
		Is there a general theorem stating which Hankel-type matrices
		of power series coefficients have this kind of geometric positivity?
		\item It is clear that free plurisubharmonic functions need not be uniformly real analytic.
		One can construct such examples using hereditary sums of polynomial identities.
		However, are all free plurisubharmonic functions real analytic? (We conjecture the answer is no.)
		\item The representation derived in Lemma \ref{representationformula} formally makes sense for real analytic functions
		which are not uniformly real analytic, however no convergence
		can be guaranteed using our methods. Does the realization formula hold anyway?
		\item Is there a direct proof of the fact that a free plurisubharmonic function is the composition of a convex function with an analytic function using the fact
		it is plurisubharmonic at each level and applying classical representation formulas?
	\end{enumerate}

\bibliography{references}
\bibliographystyle{plain}

\printindex

\end{document}